\documentclass[12pt]{amsart}
\usepackage{amssymb,amsmath}
\usepackage{geometry}    
\usepackage{mathrsfs}

\usepackage{graphicx}

\usepackage{vmargin}

\usepackage{tikz}  

\setmargrb{1in}{1in}{1in}{1in}

\newtheorem{theorem}{Theorem}[section]
\newtheorem{lemma}[theorem]{Lemma}
\newtheorem{proposition}[theorem]{Proposition}
\newtheorem{corollary}[theorem]{Corollary}
\theoremstyle{definition}

\newtheorem{remark}{Remark}

\newtheorem{conjecture}{Conjecture}
\newtheorem{problem}{Problem}

\usepackage{xcolor}

\newcommand{\la}{ \lambda_n(\xi)  }

\newcommand{\w}{w_n(\xi) }

\newcommand{\wos}{w_n^{\ast}(\xi) }
	
\newcommand{\wo}{\widehat{w}_n(\xi) }

\newcommand{\ws}{ \widehat{w}_{n}^{\ast}(\xi) }

\begin{document}

\title[On two uniform exponents of approximation]{On two uniform exponents of approximation related to Wirsing's problem}

\author{ Johannes Schleischitz}

\thanks{Middle East Technical University, Northern Cyprus Campus, Kalkanli, G\"uzelyurt \\
	johannes@metu.edu.tr ; jschleischitz@outlook.com}


\begin{abstract}
	We aim to fill a gap in the proof of
    an inequality relating two exponents of uniform Diophantine approximation stated in a paper by Bugeaud. We succeed to verify the inequality in several instances, in particular for small dimension.
    Moreover, we provide counterexamples to generalizations, which contrasts the case
    of ordinary approximation where such phenomena do not occur. Our results
    contribute to the understanding of
    the discrepancy between small absolute values of a polynomial at a given real number and approximation to the number by algebraic
    numbers of absolutely bounded degree,
    a fundamental issue in the famous problem of Wirsing and variants.
\end{abstract}

\maketitle

{\footnotesize{

Math Subject Classification 2010: 11J13, 11J82}}

\vspace{1mm}

\section{Introduction}

\subsection{Some exponents of approximation and basic properties}  \label{intro}
Let $n\ge 1$ be an integer and $\xi$ be a 
transcendental real number. For a polynomial
$P$, by its height $H_P$
we mean the maximum modulus of its coefficients. For a parameter $X>1$, 
we call the non-zero integer polynomial $P_X$ of degree at most $n$ and height at most $X$ the best approximation polynomial
for $n,\xi$ up to parameter $X$, if $P=P_X$ minimizes $|P(\xi)|$ among all such integer polynomials $P$. This is defined uniquely up to sign since $\xi$ is transcendental. Similarly, write $\alpha_X$ for the algebraic number of degree at most $n$ and height at most $X$ that minimizes $H(\alpha_{X})\cdot |\xi-\alpha_X|$. Here
the height of $\alpha$, let us denote it $H(\alpha_{X})$, is simply the height of its minimal polynomial over the integers with coprime coefficients.
Define the local exponents of approximation
at parameter $X>1$ by
\[
w(n,\xi,X)= -\frac{\log |P_X(\xi)| }{\log X}, \qquad w^{\ast}(n,\xi,X)=-\frac{\log (H(\alpha_{X})\cdot |\alpha_X-\xi|) }{\log X}
\]
and let us derive the classical exponents of approximation
\begin{equation} \label{eq:AB}
\w= \limsup_{X\to\infty} w(n,\xi,X), \qquad
\wos= \limsup_{X\to\infty} w^{\ast}(n,\xi,X)
\end{equation}
and
\begin{equation} \label{eq:BA}
\wo= \liminf_{X\to\infty} w(n,\xi,X), \qquad
\ws= \liminf_{X\to\infty} w^{\ast}(n,\xi,X).
\end{equation}
The ordinary exponents (without ``hat'')
may take the formal value $+\infty$, on the other hand
the uniform exponents (with ``hat'') are bounded independently of $\xi$ in terms of $n$, see Davenport and Schmidt~\cite{davsh}
and also~\cite{buschlei}. 
The exponents $w_n(\xi)$ date back to Mahler,
the exponents $\wos$ were introduced
by Wirsing~\cite{wirsing} in his famous paper
on approximation to real numbers by algebraic numbers of bounded degree. The uniform exponents $\wo, \ws$ were introduced by Bugeaud and Laurent~\cite{buglau}.
It is easily seen that
all four types of exponents 
from \eqref{eq:AB}, \eqref{eq:BA}
form 
non-decreasing sequences in $n$, in particular we have 
\begin{equation} \label{eq:jar}
    w_1(\xi)\le w_2(\xi)\le \cdots, \qquad
    \widehat{w}_1^{\ast}(\xi)\le \widehat{w}_2^{\ast}(\xi)\le \cdots.
\end{equation}
Several other relations are known, most notably Dirichlet's Theorem implies
\begin{equation} \label{eq:diri}
    \w\ge \wo\ge n, \qquad n\ge 1.
\end{equation}
A famous open problem of Wirsing~\cite{wirsing} asks if a strengthening of \eqref{eq:diri} of the form
$w_n^{\ast}(\xi)\ge n$ holds as well.
As a rather short argument 
(see Proposition~\ref{wellk} below)
yields for any $\xi$ the estimates
\begin{equation} \label{eq:lordi}
    \ws\le \wo, \qquad \wos\le \w,
\end{equation}
the conjectural inequality is indeed stronger,
and verified for $n\in\{1,2\}$ only so far~\cite{die2}. 
It was already noticed by Wirsing that for any $n$ and almost all real numbers
in the sense of Lebesgue measure
$\wos=n$, in fact $\ws=n$ as well as a consequence of \eqref{eq:LE} below from~\cite{buglau}, and Sprind\v{z}uk's result~\cite{sprind}. However, for general $\xi$
the best known lower bounds for $\wos$ and large $n$ of order $n/(2-\log 2)$
have very recently been obtained by Poels~\cite{poe}, improving on
$n/\sqrt{3}$ from~\cite{badsch}.
Since $\widehat{w}_1(\xi)=\widehat{w}_1^{\ast}(\xi)$ for 
any transcendental real $\xi$,
in fact both equal to $1$ for any irrational $\xi$ by Khintchine~\cite{khint},
combining \eqref{eq:jar}, \eqref{eq:diri} 
we get
\begin{equation} \label{eq:stern}
    w_n^{\ast}(\xi)\ge \widehat{w}_n^{\ast}(\xi)\ge 1, \qquad n\ge 1.
\end{equation}

Let us further define 
\[
\kappa(n,\xi,X):= w(n,\xi,X)-w^{\ast}(n,\xi,X),
\]
and accordingly
\[
\underline{\kappa}_n(\xi):= \liminf_{X\to\infty} \kappa(n,\xi,X), \quad \overline{\kappa}_n(\xi):= \limsup_{X\to\infty} \kappa(n,\xi,X).
\]
These exponents have not been explicitly defined or studied before. Note that
\begin{equation} \label{eq:ind}
\overline{\kappa}_n(\xi) \ge \max\{ \w - \wos , \wo - \ws\}
\end{equation}
and conversely
\begin{equation} \label{eq:Con}
\underline{\kappa}_n(\xi) \le \min\{ \w - \wos , \wo - \ws\}
\end{equation}
as the according estimates are true for general pairs of real functions $f(X), g(X)$ in place of
$w(n,\xi,X)$ and $w^{\ast}(n,\xi,X)$.
Note that \eqref{eq:ind}, \eqref{eq:Con}
and~\cite[Theorem~2.10]{bugdraft}
combined yield that $\underline{\kappa}_n(\xi)=\overline{\kappa}_n(\xi)=0$ for any algebraic number $\xi$.
Moreover a generalization of
both claims of \eqref{eq:lordi} of the form 
\begin{equation} \label{eq:frit}
    \underline{\kappa}_n(\xi)\ge 0,
\end{equation}
holds, by the same Proposition~\ref{wellk}.
Note however that there is no reason for 
the quantities $\underline{\kappa}_n(\xi), \overline{\kappa}_n(\xi)$ to be monotonic in $n$, in particular an analogue of \eqref{eq:jar} does not hold. We believe
that these exponents carry important 
information on the discrepancy between
small values of polynomials and approximation to real numbers by algebraic numbers, closely related to Wirsing's problem and variants, and deserve to be studied in detail.

\subsection{An open problem and related
questions}
Let us return to the classical exponents.
As stated in~\cite{bugbuch},
conversely to \eqref{eq:lordi} we have
\begin{equation} \label{eq:ordi}
    \w-\wos\le n-1.
\end{equation}
This follows from combination of Lemma~\ref{wirsing} and Lemma~\ref{feld} below. The sharpness of \eqref{eq:ordi} is open in general, for partial results 
see~\cite{b2003}.
It was further stated in~\cite{bugdraft} that we similarly have
\begin{equation} \label{eq:aaa}
    \wo-\ws\le n-1.
\end{equation}
The latter estimate is sharp, as it was noticed in~\cite{bugdraft} that any Liouville number $\xi$ (i.e. $\xi$ is transcendental and $w_1(\xi)=\infty)$ satisfies for any $n\ge 1$ the identities
\begin{equation} \label{eq:holdss}
\wo=n, \qquad \ws=1.
\end{equation}
The left identity stems from~\cite{mathematika}, see Theorem~\ref{Sch} below. In fact, it was shown in~\cite{bugdraft} that the interval $[n-2+\frac{1}{n},n-1]$
is contained in the spectrum of the difference
$\wo-\ws$, and upon assuming some natural conjecture to be true, the spectrum contains $[0,n-1]$.

However, there seems to be a considerable gap in the sketched proof
of \eqref{eq:aaa}.
Indeed, the proof is based on Lemma~\ref{feld} below, which however only holds for separable $P$. In contrast to \eqref{eq:ordi} for the ordinary exponents, where this problem
can be overcome rather easily by Lemma~\ref{wirsing} due to Wirsing below,
for uniform exponents 
it is a priori unclear why one can 
restrict to this case. While
at first sight this may appear to be a rather minor technicality, it
turns out to cause great problems, as evidenced by the results in~\S~\ref{counter} below.

The purpose of this note is
twofold: Firstly, we want to address the very problem of rigorously settling \eqref{eq:aaa}. This will be the content
of~\S~\ref{s1}.
Secondly, in~\S~\ref{se3}, we study the related parametric functions $\kappa(n,\xi,X)$ and their 
extremal values $\underline{\kappa}_n(\xi), \overline{\kappa}_n(\xi)$. We once again stress that understanding these
quantities is at the core of Wirsing's famous problem discussed in~\S~\ref{intro}.

To this end, we again 
want to emphasize that
the delicacy of proving \eqref{eq:aaa}
is reflected
in our results in~\S~\ref{counter} below,
where we will provide counterexamples
to stronger claims involving the exponents
$\kappa$ and their extremal values.
In particular, Theorem~\ref{co} below provides a counterexample of a parametric inequality in the spirit of \eqref{eq:frit} that if true would generalize both \eqref{eq:ordi}, \eqref{eq:aaa}. This illustrates that indeed one has to be very careful with non separable polynomials.

\section{Results towards a proof of \eqref{eq:aaa}} \label{s1}
Despite the examples from~\S~\ref{counter} mentioned above that suggest otherwise,
we still believe that \eqref{eq:aaa} is true for all $n$.
While we were unable to prove \eqref{eq:aaa} in full generality, we provide several weaker claims.
First we verify it for small $n$.

\begin{theorem}  \label{H}
    Claim \eqref{eq:aaa} holds for $n\le 5$ and any transcendental 
    real number $\xi$.
\end{theorem}

While the conclusion for $n\le 3$ is short by combining results from~\cite{bugbuch, mathematika}, 
the cases $n=4, n=5$ require a more
sophisticated approach.
With more effort it is likely that the range for $n$ can be extended by the underlying arguments of the latter cases. However we did not intensify our investigations in this direction, and
are doubtful that the strategy is sufficient for a proof for general $n$
without further new ideas.

Next we want to state criteria implying
\eqref{eq:aaa} for arbitrary $n$.
In this context we want to define 
another classical exponent of approximation. Let
the ordinary exponent of simultaneous approximation denoted by $\la$ be the supremum of $\lambda$ so that
\[
\max_{1\le j\le n} |q\xi^j -p_j|\le q^{-\lambda}
\]
has infinitely many solutions in integer
vectors $q,p_1,\ldots,p_n$.

\begin{theorem} \label{B}
    Let an integer $n\ge 1$ and a transcendental real number $\xi$ satisfy any of the following conditions
    \begin{itemize}
        \item[(i)] $\wo=n$
        \item[(ii)] $\wo>2n-3$
        \item[(iii)] $\ws\ge n-2$
        \item[(iv)] $\lambda_n(\xi)\in [\frac{1}{n},\frac{1}{n-2}]\cup (1,\infty]$
        \item[(v)] $w_1(\xi)\ge n$
        \item[(vi)] $w_n(\xi)\le \frac{2n-1+\sqrt{4n-3}}{2}$
    \end{itemize}
    Then \eqref{eq:aaa} holds.
\end{theorem}

It has very recently
been shown by Poels~\cite[Theorem~1.1]{poels} that
$\widehat{w}_3(\xi)\le 2+\sqrt{5}$ and
$\wo\le 2n-2$ for $n\ge 4$, improving for $3\le n\le 9$ on~\cite{ichacta}. 
Additionally~\cite[Theorem~1.2]{poels} shows that (ii) can only happen
for small $n$, as for sufficiently large
$n$ indeed stronger bounds of order
$\wo\le 2n-\sqrt[3]{n}/3$ were obtained.
See also~\cite{icharc}.

Notably, for $n\ge 3$ it is in fact not known if counterexamples to (i) exist. 
By \eqref{eq:diri}, one of the conditions (i), (ii) in particular applies to any $\xi$ if $n\le 3$ (and (iii) does as well), so part of Theorem~\ref{H} is directly implied. 
However, the claim for $n=2$ from Theorem~\ref{H} is in fact used in the proof of this case. 

A corollary of part (ii) is the following unconditional but weaker bound that may be of some interest for small $n$ especially.

\begin{theorem}
    For any integer $n\ge 3$ and any transcendental real number $\xi$ we have
    \[
    \wo-\ws \le 2n-4.
    \]
\end{theorem}

\begin{proof}
    If $\wo>2n-3$ then the stronger bound $\wo-\ws\le n-1\le 2n-4$ for $n\ge 3$ applies by Theorem~\ref{B}, (ii). If otherwise $\wo\le 2n-3$ then
    \eqref{eq:stern} implies
    $\wo-\ws\le \wo-1\le 2n-4$ as well.
\end{proof}

This slightly improves on the bound $2n-2$
that is immediate from Davenport and Schmidt~\cite{davsh}, in fact $2n-3$
for $n\ge 4$ 
follows from aforementioned~\cite[Theorem~1.1]{poels}.
Note that for large $n$ we infer the bound $2n-\sqrt[3]{n}/3-1$ from~\cite[Theorem~1.2]{poels} recalled above as well,
in particular the expression is not of order $2n-O(1)$ as $n\to\infty$. It would be desirable to obtain a bound
of the form $(2-\epsilon)n$ for some explicit $\epsilon>0$.

\section{On the exponents $\kappa(n,\xi,X)$}  \label{se3}

\subsection{On $\overline{\kappa}_n$} \label{s2}

Our next theorem gives another upper estimate for the difference $\wo-\ws$, 
in fact for the quantity $\overline{\kappa}_n(\xi)$.
However another exponent of approximation occurs in the formula.
Write $\lfloor x\rfloor$ for the largest 
integer less or equal than $x$. 

\begin{theorem} \label{A}
    Let $n\ge 2$ be an integer and derive
    \[
    k= \left\lfloor \frac{n}{2}\right\rfloor.
    \]
    Then for any real number $\xi$ we have
    \begin{equation} \label{eq:schwach}
    \wo - \ws \le \max\{ n-1 , w_{k}(\xi)  \}.
    \end{equation}
    In fact the stronger inequality
\begin{align}
\overline{\kappa}_n(\xi) \le \max\{ n-1 , w_{k}(\xi)  \}  \label{eq:aign}
\end{align}
holds.
\end{theorem}

By \eqref{eq:ind}, estimate \eqref{eq:aign}
indeed implies \eqref{eq:schwach}.
The claim shows that the desired estimate \eqref{eq:aaa} holds if $w_k(\xi)\le n-1$.
The classical formula of Bernik~\cite{bernik}
\begin{equation} \label{eq:bernik}
    \dim_H(\{ \xi: w_m(\xi)\ge \lambda\}) = \frac{m+1}{\lambda+1},\qquad m\ge 1,\; \lambda\in [m,\infty],
\end{equation}
where $\dim_H$ denotes Hausdorff dimension, applied for $m=k$
in particular implies 

\begin{corollary} \label{nko}
    Given an integer $n\ge 2$, define the set of counterexamples
    to \eqref{eq:aaa} as
\[
\Omega_n:= \{ \xi\in\mathbb{R}: \wo-\ws>n-1\}
\]
and let similarly
\[
\Gamma_n:= \{ \xi\in\mathbb{R}: \overline{\kappa}_n(\xi)>n-1\}\supseteq \Omega_n.
\]
Then
\[
\dim_H(\Omega_n)\le \dim_H(\Gamma_n)\le \frac{1}{2}+o(1), \qquad n\to\infty.
\]
\end{corollary}

By Theorem~\ref{H}, the estimates for $\Omega_n$ are only relevant for $n\ge 6$.
%
%
While the bound in Corollary~\ref{nko} 
seems not particularly strong,
it is not clear if there is any other easy argument available that improves on it,
even for $\Omega_n$.
Indeed, it seems no improvement via
Theorem~\ref{B} can be obtained when combining it with known metrical theory for classical exponents, and only (i), (iii) have potential to lead to such an improvement if the metrical
theory is developed well enough (for (iv)
this is excluded for $n\ge 6$ by a metrical result from~\cite{bere}). In this context note that the Hausdorff dimension of counterexamples to (iii) is at least $1/4-o(1)$ as $n\to\infty$, as a consequence 
of~\cite[Theorem~2.5]{equprin} (in fact of its proof) and \eqref{eq:bernik}.

For a reverse positive lower bound for the Hausdorff dimension of $\Gamma_n$, see Corollary~\ref{lkor} below. For $n\in \{2,3\}$ we will improve the bounds resulting from Corollary~\ref{nko} in Corollary~\ref{Kor} below.

\subsection{Counterexamples to stronger versions of \eqref{eq:aaa} }  \label{counter}
The following example shows that a parametric
version of \eqref{eq:aaa} does not hold in general, so that $w_k(\xi)$ cannot be dropped in \eqref{eq:aign} of Theorem~\ref{A}. Recall that a Liouville number
is a transcendental real number satisfying
$w_1(\xi)=\infty$.

\begin{theorem} \label{co}
    Let $n\ge 2$ be an integer and $\xi$ be a transcendental real number. 
    Let $2\le k\le n$ be another integer. If $w_1(\xi)>n+k-1$ then
    \begin{equation}  \label{eq:apply}
        \overline{\kappa}_n(\xi)\ge \left(1-\frac{1}{k}\right)\cdot w_1(\xi).
    \end{equation}
    Thus,
    if $w_1(\xi)>2n-1$, then
    \begin{equation} \label{eq:apply2}
    \overline{\kappa}_n(\xi)\ge \left(1-\frac{1}{n}\right)\cdot w_1(\xi).
    \end{equation}
    In particular, if $\xi$ is a Liouville number then
    \begin{equation} \label{eq:apply3}
    \overline{\kappa}_n(\xi) = \infty.
    \end{equation}
\end{theorem}

The proof is based on 
Liouville's inequality Theorem~\ref{liouville} below on the minimum distance between algebraic numbers.
As another application of Theorem~\ref{co}, we get

\begin{corollary} \label{lkor}
    We have
    \begin{equation} \label{eq:bulow}
    \dim_H(\Gamma_n)\ge \dim_H(\Gamma_n\setminus \Omega_n)\ge \frac{2-o(1)}{n+\sqrt{n}}, \qquad n\to\infty,
\end{equation}
with $\Gamma_n$
defined in Corollary~\ref{nko}. Moreover
we always have $\dim_H(\Gamma_n\setminus \Omega_n)\ge 1/n$.
\end{corollary}

\begin{proof}
Let 
$k$ be the smallest integer with $k^2-k+1>n$, of order $\sqrt{n}$. Then
the right hand side of \eqref{eq:apply}
exceeds $n-1$ as soon as $w_1(\xi)>n+k-1$, which also justifies 
application of \eqref{eq:apply}.
Then we can further use \eqref{eq:bernik} for $m=1$ and any $\lambda>n+k-1$, which gives the lower bound of \eqref{eq:bulow} for the set $\Gamma_n$. Finally the hypothesis implies \eqref{eq:aaa} for all $\xi$ within our set from \eqref{eq:bernik} by Theorem~\ref{B}, (v), hence
the estimate is preserved when removing $\Omega_n$.
The bound $2/(2n)=1/n$
follows similarly from \eqref{eq:apply2}
with $\lambda=2n-1+\varepsilon$ via \eqref{eq:bernik}.
\end{proof}

 Similarly parametric variants 
 of \eqref{eq:bulow} for estimating
 from below
 \begin{equation} \label{eq:frab}
\dim_H(\{ \xi\in\mathbb{R}: \overline{\kappa}_n(\xi) > \lambda \}), \qquad \lambda\ge \lambda_0(n)
\end{equation}
with large
enough $\lambda_0(n)$ can be obtained. Moreover,
variants of Theorem~\ref{co} involving $w_{\ell}(\xi)$ for any $\ell\le n/2$ can be derived, in particular we have a generalization of \eqref{eq:apply3} of the form
$\overline{\kappa}_n(\xi) = \infty$
as soon as $\xi$ is a $U_{\ell}$-number of index $\ell\le n/2$ in Mahler's classification of real numbers, meaning that $w_{\ell-1}(\xi)<\infty$ and $w_{\ell}(\xi)=\infty$.

As indicated in~\S~\ref{intro}, inseparable polynomials
cause problems when trying to prove
\eqref{eq:aaa}, the deeper reason being that we cannot apply Lemma~\ref{feld} and  Lemma~\ref{sepp} below. This motivates to define the separable local exponent 
$w_{sep}(n,\xi,X)$ like $w(n,\xi,X)$ but restricting to separable
polynomials $P$. Let us derive
the uniform exponent
\[
\widehat{w}_{n,sep}(\xi):=
\liminf_{X\to\infty} w_{sep}(n,\xi,X).
\]
Hence, in order to prove \eqref{eq:aaa},
one may try to settle for any real $\xi$
the identity
\begin{equation} \label{eq:falsch}
\widehat{w}_{n,sep}(\xi)= \wo.
\end{equation}
Indeed, if \eqref{eq:falsch} holds for
some $n,\xi$, then \eqref{eq:aaa} holds
for this pair via 
\begin{align*}
\wo-\ws &= \widehat{w}_{n,sep}(\xi) - \ws =  \liminf_{X\to\infty} w_{sep}(n,\xi,X)-\liminf_{X\to\infty} w^{\ast}(n,\xi,X) \\ &\le \limsup_{X\to\infty} (w_{sep}(n,\xi,X)-w^{\ast}(n,\xi,X)) \le n-1,
\end{align*}
where for the last inequality we used Lemma~\ref{feld} below. For completeness, define similarly
$w_{irr}(n,\xi,X)$ and
$\widehat{w}_{n,irr}(\xi)$
when restricting to irreducible polynomials. It is clear that
\[
w_{irr}(n,\xi,X)\le w_{sep}(n,\xi,X)
\le w(n,\xi,X),
\]
hence for any $n, \xi$ we have
\[
\widehat{w}_{n,irr}(\xi)\le \widehat{w}_{n,sep}(\xi) \le \wo.
\]
We remark that
one may accordingly define ordinary exponents 
$w_{n,irr}(\xi), w_{n,sep}(\xi)$ as well
via upper limits,
however they turn out to be equal to $\w$ by Lemma~\ref{wirsing} below.

However, at least for $n=2$ we can prove
that \eqref{eq:falsch} is in general false as well, so this avenue is not possible.

\begin{theorem}  \label{liou}
Assume $\xi$ satisfies $w_1(\xi) > 3$. Then
\begin{equation} \label{eq:zwe}
\widehat{w}_{2,sep}(\xi) \le 1+\frac{2}{w_1(\xi)-1} < 2 = \widehat{w}_2(\xi). 
\end{equation}
 In particular, for any Liouville number $\xi$ we have
    \[
    \widehat{w}_{2,irr}(\xi) =\widehat{w}_{2,sep}(\xi) = 1.
    \]
\end{theorem}

\begin{remark}
    A slightly stronger estimate for $\widehat{w}_{2,irr}(\xi)$ 
    upon $w_1(\xi)\in (3,\infty)$ can be 
    deduced from the method. We further remark that as shown in~\cite{moscj}, the exponent
    $\widehat{w}_{=n,irr}(\xi)$ (denoted
    just $\widehat{w}_{=n}(\xi)$ in~\cite{moscj}) with 
    polynomials irreducible of degree
    precisely $n$ equals $0$
    for any Liouville number and any $n\ge 2$. The proof is similar to our proof 
    of Theorem~\ref{liou} below, based on Minkowski's Second Convex Body Theorem.
    In exact degree the associated ordinary 
    exponents $w_{=n,irr}(\xi)$ become more interesting as well. In particular, it is open for $n>7$ if it is always at least $n$ (for $n\le 7$
    this was settled in~\cite[Theorem~1.3]{newp}).
\end{remark}

We believe there is equality in the most left inequality of \eqref{eq:zwe}. Denoting $spec(.)\subseteq \mathbb{R}$ the set
of values taken by an exponent when inserting all real numbers,
this would imply
\begin{equation} \label{eq:lere}
spec(\widehat{w}_{2,sep})= [1,2]\cup spec(\widehat{w}_2), \qquad spec(\widehat{w}_{2}-\widehat{w}_{2,sep})=[0,1].
\end{equation}
The spectrum of $\widehat{w}_2$ is known to be contained and dense in $[2,(3+\sqrt{5})/2]$. The inclusion
of $[1,2]$ in the left identity of \eqref{eq:lere} follows directly
from equality in \eqref{eq:zwe}. On the other hand,
if $\widehat{w}_2(\xi)>2$ for some $\xi$ then it follows as in~\cite{moscj}
that all best approximations are irreducible of degree exactly two, which implies
$\widehat{w}_{2,sep}(\xi)=\widehat{w}_2(\xi)$. This argument
shows that the spectrum
of $\widehat{w}_{2,spec}(\xi)$ contains the spectrum of $\widehat{w}_2(\xi)$ as well, more precisely unconditionally we have
\[
spec(\widehat{w}_2)\cap (2,\infty)= spec(\widehat{w}_{2,sep})\cap (2,\infty).
\]
The point $\{2\}$ can be added by considering generic numbers with $w_2(\xi)=2$, see Theorem~\ref{str} below. Hence the left identity  of \eqref{eq:lere} holds upon our assumption.
The right identity of \eqref{eq:lere} uses additionally that $w_1(\xi)\ge 3>2$ implies $\widehat{w}_2(\xi)=2$ by Theorem~\ref{Sch} below.

We further believe that for larger $n$
again Liouville numbers will provide counterexamples to \eqref{eq:falsch}, however the proof will be more involved so we do not attempt to make it rigorous.

\subsection{Some remarks on $\underline{\kappa}_n$ and the limit of $\kappa$}

Let us study
$\underline{\kappa}_n(\xi)$.
By \eqref{eq:frit}, \eqref{eq:ordi}, \eqref{eq:Con} we get
\begin{equation} \label{eq:kk}
0\le \underline{\kappa}_n(\xi)\le \w - \wos \le n-1.
\end{equation}
This motivates the following

\begin{problem}  \label{pp}
    Can we improve the bound $n-1$ 
    for $\underline{\kappa}_n$
    from \eqref{eq:kk} for $n\ge 2$ and
    all $\xi$? Is it in fact always $0$?
\end{problem}

Let us first consider $n=2$. Then, 
complementary to the bound $n-1=1$
in \eqref{eq:kk}, we get
  the following estimates in terms 
    of the uniform exponent.

    \begin{theorem} \label{n=2}
        For any real number $\xi$ we have 
        \begin{equation} \label{eq:EST1}
        \underline{\kappa}_2(\xi) \le \frac{w_2(\xi)}{\widehat{w}_2(\xi)-1} - \widehat{w}_2(\xi)
        \end{equation}
    and
        \begin{equation} \label{eq:EST2}
        \underline{\kappa}_2(\xi) \le \frac{1}{\widehat{w}_2(\xi)-1}.
        \end{equation}
    \end{theorem}

    The latter \eqref{eq:EST2}
    is a refinement of \eqref{eq:kk}
    by \eqref{eq:diri}.
    Note that the bound in \eqref{eq:EST1} is always 
    non-negative~\cite{jarnik}. We enclose a few observations on Problem~\ref{pp} for general $n$:

\begin{itemize}
    \item If $\underline{\kappa}_n(\xi)=0$ for all $\xi$, this
    would imply
    Wirsing's problem has an affirmative answer for given $n$. In fact $\underline{\kappa}_n(\xi)=0$ implies
    $\wos\ge \wo\ge n$. While this 
    is audacious to conjecture in general, for $n=2$,
    Theorem~\ref{n=2} and the results from~\cite{die2, mos} it is based on indicate that it may be true.
    \item Constructing $\xi$ with $\underline{\kappa}_n(\xi)>0$ does
    not seem easy either, as {\em all}
    best approximations of large norm must have
    untypically small derivative. Moreover it probably forces $\wo>n$, see Conjecture~\ref{konsch} below.
    \item An analogue of Jarn\'ik's estimate on the logarithmic ratio of best approximation norms,
    as in the proof of Theorem~\ref{n=2} below,
    for higher dimension would 
    generalize Theorem~\ref{n=2}. However, this seems to this date unknown. While
    the results from~\cite{mamo} suggest
    such an effective estimate, 
    it seems not settled. 
    Nevertheless,
    in view of Conjecture~\ref{konsch} below, this suggests
    that $n-1$ can be improved. 
    Moreover,
    it seems that Wirsing's method and results~\cite{wirsing} 
    do not immediately imply anything better than $n-1$ either. For example if $w_n(\xi)=2n-1$ (or larger).    
\end{itemize}

The next metrical result on the limit of $\kappa$ is an easy consequence of a claim from~\cite{buglau}.

\begin{theorem}  \label{str}
    Let $n\ge 2$ be an integer 
    and $\xi$ be a transcendental real number.
    If
    \[
    w_n(\xi)=n
    \]
    then
    \begin{equation} \label{eq:GH}
    \lim_{X\to\infty} \kappa(n,\xi,X)=0.
    \end{equation}
    In particular we have \eqref{eq:GH} for
    almost all $\xi$ with respect to Lebesgue measure.
\end{theorem}

\begin{proof}
We use the estimate of Bugeaud and Laurent~\cite[Theorem~2.1]{buglau}
\begin{equation} \label{eq:LE}
\ws \ge \frac{\w}{\w-n+1}.
\end{equation}
It follows that
\begin{align}  \label{eq:Ff}
\overline{\kappa}_n(\xi) &=
\limsup_{X\to\infty} (w(n,\xi,X) - w^{\ast}(n,\xi,X)) \nonumber \\
&\le \limsup_{X\to\infty} w(n,\xi,X) - \liminf_{X\to\infty} w^{\ast}(n,\xi,X)
\nonumber \\ &= w_n(\xi) - \ws \le \w - \frac{\w}{\w-n+1}. 
\end{align}
For $\w=n$ the right hand side vanishes,
thus together with \eqref{eq:frit} we conclude
\[
0\le \underline{\kappa}_n(\xi)\le \overline{\kappa}_n(\xi)\le 0
\]
so the limit is $0$.
Finally for the metrical claim
we conclude with Sprind\v{z}uk's famous result~\cite{sprind}.
\end{proof}

More generally
combining \eqref{eq:Ff} and \eqref{eq:bernik} enables us to estimate from above \eqref{eq:frab}
for $\lambda\ge 0$,
complementary to Corollary~\ref{nko} where
$\lambda=n-1$. For large $\lambda$, only small improvements compared to the bounds obtained via trivially estimating $\ws\ge 1$ in place of \eqref{eq:LE} and using \eqref{eq:bernik}
are achieved. In particular for $\lambda=n-1$ the bound becomes significantly weaker than the one from Corollary~\ref{nko}.

The proof of Theorem~\ref{str} and the dual inequality $\wos\ge \wo/(\wo-n+1)$ to \eqref{eq:LE} established in~\cite{buglau} as well suggest the following may be true.

\begin{conjecture}  \label{konsch}
    We have
\begin{equation} \label{eq:rororo}
    \underline{\kappa}_n(\xi) \le
    \wo- \frac{\wo}{\wo-n+1}.
\end{equation}
In particular
\[
\wo=n
\]
implies
\[
\underline{\kappa}_n(\xi)= 0.
\]
\end{conjecture}

However, the conclusion even for the special latter claim is not as straightforward as for Theorem~\ref{str},
and we could not prove it rigorously.
Note also that $\wo=n$ does not imply
the stronger claim $\wo-\ws=0$, as any Liouville number $\xi$ satisfies \eqref{eq:holdss}. Combining \eqref{eq:EST2} with our unproved
estimate \eqref{eq:rororo} for $n=2$ gives a conditional bound
\[
\underline{\kappa}_2(\xi) \le \frac{1}{\sqrt{2}}.
\]

\section{Preparatory lemmas}

We need some lemmas. The first is essentially
due to Wirsing~\cite[Hilfssatz 4]{wirsing}.

\begin{lemma}[Wirsing]  \label{wirsing}
    Let $\xi$ be a transcendental real number and $n\ge 1$. Let $\epsilon>0$. Let 
    $P$ be an integer polynomial of degree at most $n$ and
    large enough height, and assume for some $w>0$ we have 
    \[
    |P(\xi)|\le H_P^{-w}.
    \]
    Then there exists an irreducible divisor $Q$ of $P$ such that
    \[
    |Q(\xi)|\le H_Q^{-w+\epsilon}.
    \]
    We can assume $H_Q\to\infty$ as $H_P\to\infty$.
\end{lemma}

We note that this easily implies \eqref{eq:ordi}.
The next is part of Lemma~A8 in Bugeaud's book~\cite{bugbuch} attributed to Feldman.

\begin{lemma}[Feldman] \label{feld}
    If $P$ is a separable integer polynomial of degree at most $n$ then it has a root $\alpha$ such that
    \[
     |\xi-\alpha| \ll_n |P(\xi)|\cdot H_P^{n-2}.
    \]
\end{lemma}

Thus
without restriction to separable polynomials, the claim \eqref{eq:aaa} would be obvious. The next well-known claim is the essence in the proof of Lemma~\ref{wirsing} above.

\begin{lemma}[Gelfond's Lemma]  \label{gel}
    For any polynomials with $P=QR$
    of degree at most $n$ we have
    \[
    H_P \asymp_n H_QH_R.
    \]
\end{lemma}

We state Liouville's inequality, see~\cite[Corollary~A2]{bugbuch}.

\begin{theorem}[Liouville inequality] \label{liouville}
    If $\alpha$ is real algebraic of degree $m$ and $\beta$ is real algebraic of degree $n$ and
    they have different minimal polynomials then
    \[
    |\alpha-\beta| \gg_{m,n} H(\alpha)^{-n} H(\beta)^{-m}.
    \]
\end{theorem}

 The next is~\cite[Theorem~2.3]{buschlei}.

\begin{theorem}[Bugeaud, Schleischitz] \label{BuS}
    For any positive integers $m,n$ and any real transcendental number $\xi$, we have
    \[
    \min\{ w_m(\xi), \wo\} \le m+n-1.
    \]
\end{theorem}

The next  
is combination 
of~\cite[Theorem~1.12]{mathematika} and~\cite[Theorem~1.6]{mathematika}.

\begin{theorem}[Schleischitz] \label{Sch}
Let $n\ge 1$ be an integer and $\xi$ be a transcendental real number. 
    If $w_1(\xi)\ge n$ then $\wo=n$.
    If $\la>1$ then $w_1(\xi)> 2n-1\ge n$, so again $\wo=n$.
\end{theorem}

For completeness, 
  we end this section with a well known fact although not explicitly used in proofs below, see for example~\cite[Proposition~2.7]{newp} for a proof.

    \begin{proposition}  \label{wellk}
        For any integer polynomial $P$ 
        of degree at most $n$, any of its roots $\alpha$ satisfies
        \[
        |\xi-\alpha|\gg_{n,\xi} H(\alpha)^{-1}|P(\xi)|. 
        \]
    \end{proposition}

Proposition~\ref{wellk} immediately proves \eqref{eq:lordi}, \eqref{eq:frit}.

For the proofs below
recall that $P=P_X$ minimizes $|P(\xi)|$ among all integer polynomials of degree at most $n$ and height at most $X$.
It satisfies
\[
|P_X(\xi)|\le X^{-\wo+o(1)}, \qquad X\to\infty.
\]

\section{Proof of Theorem~\ref{A} }
    Let $\epsilon>0$.
Let $X$ be any large parameter and $P=P_X$ 
of height $H_P\le X$ be the
according best approximation polynomial.
For simplicity write $\tilde{w}=w(n,\xi,X)$.
Define $\sigma=\sigma_X$ implicitly by
\[
X= H_P^{\sigma}, \qquad \sigma=\sigma_X\ge 1.
\]
Then
\[
|P(\xi)| = H_P^{-\sigma \tilde{w}}.
\]
By Wirsing's Lemma~\ref{wirsing} and again assuming $X$ was chosen sufficiently large, there is an irreducible (thus separable) divisor $Q=Q_X$ of $P$ so that
\[
|Q(\xi)| = H_Q^{-\sigma w}, \qquad w\ge \tilde{w}-\epsilon.
\]
We can assume $H_Q>1$ and therefore may
define $\tau=\tau_X$ implicitly via
\[
H_Q^{\tau} = X, \qquad \tau=\tau_X\ge 1-\epsilon.
\]
The lower bound is obtained from Gelfond's Lemma~\ref{gel} which in fact implies the stronger claim $H_Q\ll_n H_P\le X$. 
Let $m=m_X=\deg Q\le n$ and $r=n-m=\deg R$ where $P=QR$. 

We may assume 
\begin{equation}  \label{eq:mk}
    m>k=\left\lfloor \frac{n}{2}\right\rfloor
\end{equation}
 for large enough $X$. Indeed otherwise if $m\le k$ for arbitrarily large $X$ then
 for some modified small $\varepsilon>0$ we get
 \[
 w_k(\xi)\ge w_m(\xi)\ge \sigma\tilde{w}-\varepsilon\ge \tilde{w}-\varepsilon,
 \]
 which further implies
\[
\tilde{w} \le w_k(\xi)+\varepsilon \leq \max\{ w_k(\xi), n-1\} + w^{\ast}(n,\xi,X)+\varepsilon,
\]
and the claim follows by subtracting
$w^{\ast}(n,\xi,X)$ from both sides
as $X\to\infty$
and thus $\varepsilon\to 0$.

To simplify the presentation, let us from now on consider a sequence of $P$ and assume $H_P\to\infty$ implying
$X\to\infty$, and use $o(1)$ notation as $X\to\infty$.
 Similar to the proof of Theorem~\ref{n=2},
 Feldman's Lemma~\ref{feld} yields that the polynomial $Q$ has a root $\alpha$ so that
 \[
|\alpha-\xi| \ll_n H(\alpha)^{-1} |Q(\xi)| H_Q^{m-1}=
H(\alpha)^{-1} X^{-\sigma w/\tau }X^{(m-1)/\tau}=
H(\alpha)^{-1} X^{-\frac{\sigma w-m+1}{\tau} }.
 \]
Hence
\begin{equation} \label{eq:b1}
w^{\ast}(n,\xi,X)\ge \frac{\sigma w-m+1}{\tau}-o(1), \qquad X\to\infty.
\end{equation}
On the other hand, if we assume $H_R\to\infty$ as $X\to\infty$, we get
\[
|P(\xi)|= |Q(\xi)|\cdot |R(\xi)| \ge
X^{-\sigma w/\tau}\cdot H_R^{- w_r(\xi)-o(1)},\qquad X\to\infty.
\]
Otherwise if $H_R$ is bounded on a subsequence as $X\to\infty$ then by transcendence of $\xi$ we get stronger estimates from $|R(\xi)|\gg 1$.
Now since by Gelfond's Lemma~\ref{gel}
\[
H_R\asymp_n \frac{H_P}{H_Q}= X^{1/\sigma}\cdot X^{-1/\tau}= X^{\frac{\tau-\sigma}{\sigma\tau} }
\]
 we infer
 \begin{equation} \label{eq:b2}
 w(n,\xi,X)\le  \frac{\sigma w}{\tau} + w_r(\xi)\frac{\tau-\sigma}{\sigma\tau} + o(1), \qquad X\to\infty.
 \end{equation}
 Note that 
 \[
 r=n-m\le n-k-1\le k
 \]
 by \eqref{eq:mk} and definition of $k$. Together with \eqref{eq:jar} this implies
 \[
  w_r(\xi)\le w_k(\xi).
 \]
 Hence comparing the bounds
 \eqref{eq:b1}, \eqref{eq:b2} we get
 \begin{align} \label{eq:tpr}
 \kappa(n,\xi,X)&=w(n,\xi,X)-w^{\ast}(n,\xi,X)\le w_r(\xi)\frac{\tau-\sigma}{\sigma\tau} + \frac{m-1}{\tau}+o(1)\nonumber \\ &\le w_k(\xi)\frac{\tau-\sigma}{\sigma\tau} + \frac{n-1}{\tau}+o(1). 
 \end{align}
 Hereby we consider $\tau\ge 1, \sigma\ge 1$ 
 absolute so that
 the expression is maximized upon all pairs
 $\ge 1$.
 The right hand side clearly decays in $\sigma$, hence we can put $\sigma=1$ and get
 \[
\kappa(n,\xi,X)\le \frac{1}{\tau} ((\tau-1)w_k(\xi)+n-1)+o(1)=
w_{k}(\xi) + \frac{ n-1-w_k(\xi) }{\tau}+o(1).
 \]
 If $n-1\ge w_k(\xi)$ then the bound decreases in $\tau$, hence we may put $\tau=1$ for all large $X$ to get the bound $n-1$. If otherwise $n-1\le w_k(\xi)$ then we have the bound $w_k(\xi)$. Since $X$ was arbitrary, the stronger claim \eqref{eq:aign} follows as $X\to\infty$ and \eqref{eq:schwach} is implied by \eqref{eq:ind}. 

\section{Proof of Theorem~\ref{B}}

 \subsection{Another lemma}

We first verify 

\begin{lemma}  \label{sepp}
    Assume for given $n,\xi$ and for all large $X$, the polynomial
    $P_X$ is separable. Then \eqref{eq:aaa} holds, in fact 
    $\overline{\kappa}_n(\xi)\le n-1$. More generally
    for every $X$ with this property
    \[
    \kappa(n,\xi,X)\le
    n-1.
    \]
\end{lemma}

\begin{proof}
    We may apply Lemma~\ref{feld} to $P_X$.
    Proceeding as in the proof of Theorem~\ref{A} (or Theorem~\ref{n=2}) with $P=Q$ and thus $\sigma=\tau$, and writing
    $w=w(n,\xi,X)$,
    the polynomial $P_X$ 
    of degree $m\le n$ has a root $\alpha$ with
    \[
    |\xi-\alpha| \ll_n H(\alpha)^{-1} X^{-\frac{\sigma w-m+1}{\sigma} }.
    \]
    In other words
    \[
    w^{\ast}(n,\xi,X)\ge \frac{\sigma w-m+1}{\sigma}\ge w-\frac{n-1}{\sigma}\ge w-n+1.
    \]
    The latter claim is now immediate,
    and the first follows by \eqref{eq:ind}
    if the separability hypothesis holds for all large $X$.
\end{proof}

We remark that thanks to this lemma we can improve the bound of Corollary~\ref{nko} for very small $n$.

\begin{corollary}  \label{Kor}
    We have
    \[
0=\dim_H(\Omega_2)\le \dim_H(\Gamma_2)\le \frac{2}{3}
\]
and
\[
0=\dim_H(\Omega_3)\le \dim_H(\Gamma_3)\le \frac{1}{2}.
\]
\end{corollary}

From the argument of Corollary~\ref{nko}
we would get the bounds $1$ and $2/3$ respectively.

\begin{proof}
    All but the upper bounds 
    for $\dim_H(\Gamma_2), \dim_H(\Gamma_3)$
    are clear 
    by Theorem~\ref{H}. If for all large $X$ the best approximation $P_X$ 
    is separable, then
    Lemma~\ref{sepp}
    implies $\overline{\kappa}_n(\xi)\le n-1$, so we can assume the opposite. But for $n\le 3$, this means that $P_X$ decomposes into linear factors for certain arbitrarily large $X$. This
    implies $w_1(\xi)\ge n$ by Wirsing's Lemma~\ref{wirsing}, and the estimates follow from \eqref{eq:bernik} applied to $m=1$.
\end{proof}

\subsection{Proof of Theorem~\ref{B}} 
If $\wo=n$ then by \eqref{eq:stern} trivially
\begin{equation} \label{eq:triv}
\ws\ge 1= n-n+1= \wo-n+1.
\end{equation}
%

Now assume $\wo>2n-3$. 
By Theorem~\ref{H} we may assume $n\ge 3$ (in fact $n\ge 4)$.
Assume 
$P_X$ is not separable for some large $X$. Then it contains 
a square of a polynomial as a factor. 
But by $n\ge 3$ this means that
any irreducible factor has degree
at most 
$\max\{ \lfloor n/2\rfloor,n-2\}= n-2$.
Thus Wirsing's Lemma~\ref{wirsing} 
and \eqref{eq:diri} imply
\[
w_{n-2}(\xi)\ge \wo\ge n.
\]
However by Theorem~\ref{BuS} this implies
\[
\wo= \min\{ \wo, w_{n-2}(\xi)\}\le n+(n-2)-1=2n-3.
\]
Thus assuming the reverse inequality 
$\wo>2n-3$ all best approximation polynomials of large norm are separable. Thus the claim 
follows from Lemma~\ref{sepp}.

Now assume $\ws\ge n-2$. Then the claim is 
obviously true if $\wo\le (n-2)+(n-1)=2n-3$.
If this is false then the implication follows from (ii).

Assume (iv) holds. If $\la\le 1/(n-2)$
then by the main result of~\cite{ichdeb} 
\[
\ws\ge \frac{1}{\la}\ge n-2
\]
which is hypothesis (iii). If $\la>1$, then
it follows 
from~ Theorem~\ref{Sch} 
$\wo=n$, so we reduced it to (i).
This also contains the implication from (v).

Finally for the conclusion from (vi), 
combining \eqref{eq:Ff} and \eqref{eq:ind}
yields
\[
\wo-\ws\le \overline{\kappa}_n(\xi) \le \w - \frac{\w}{\w-n+1}. 
\]
A short calculation shows that
the right hand side is at most $n-1$ 
if $\w$ is bounded as in (vi).

\begin{remark} \label{RR}
    While we used Theorem~\ref{H} in the deduction
    from some conditions, it was not
    applied when concluding from (v).
    This is important to note 
    as we will conversely
    use this very implication from Theorem~\ref{B}, (v) in the
    deduction of Theorem~\ref{H} below,
    so there is no circular reasoning
    in our argument.
\end{remark}

\section{Proof of Theorem~\ref{H}}

\subsection{Proof for $n\le 3$}

Let $n\le 3$. Then if $P_X$ is not separable,
it decomposes into linear factors. 
If this happens for certain
arbitrarily large $X$, by
Wirsing's Lemma~\ref{wirsing} and \eqref{eq:diri} this implies
that
\[
w_1(\xi) \ge \wo\ge n.
\]
We conclude 
via Theorem~\ref{B}, (v) in view
of Remark~\ref{RR}. 
If otherwise there
is no such sequence of $X$ that tends to infinity, then
\eqref{eq:aaa} holds as well by Lemma~\ref{sepp}.

\subsection{Towards the proof for $n\in\{4,5\}$. A preparatory lemma} \label{nota}

The cases $n=4, n=5$ are considerably more complicated.
We prove a general auxiliary fact for each $n$.

Let us first fix some notation.
Let $X>1$ be any large number and
$P=P_X$ the according best approximation polynomial.
Write
\[
P= \prod_{i=1}^{\ell} Q_i^{\alpha_i}, \qquad 
\alpha_i\in \mathbb{N}
\]
for its factorization into irreducible polynomials. We may assume $H_i:=H(Q_i)>1$ for all $i$, as otherwise if some fixed  $Q$ occurs as a factor of $P_X$ 
for arbitrarily large $X$, we can instead of $P_X$ consider $P_X/Q$, of comparable height and evaluation at $\xi$, using for the latter that $\xi$ is transcendental, in the below argument. We omit the details.
Let further
\[
d_i= \deg Q_i \ge 1.
\]
Note that
\begin{equation} \label{eq:side1}
\sum_{i=1}^{\ell} \alpha_i  d_i \le n.
\end{equation}
Let $\gamma_i>0$ be defined by 
\[
|Q_i(\xi)| = H_i^{-\gamma_i}.
\]
Define $\beta_i>0$, $1\le i\le \ell$, by
\[
H_i= X^{\beta_i}. 
\]
Gelfond's Lemma~\ref{gel} implies $\prod X^{\alpha_i\beta_i} = \prod H_i^{\alpha_i}\ll_n H_P\le X$ and thus
\begin{equation} \label{eq:side2}
\sum_{i=1}^{\ell} \alpha_i \beta_i \le 1+o(1).
\end{equation}
All error terms are understood as $X\to\infty$. Now since
\[
|P(\xi)| = \prod_{i=1}^{\ell} |Q_i(\xi)|^{\alpha_i}= \prod_{i=1}^{\ell} H_i^{-\alpha_i\gamma_i}= X^{-\sum \alpha_i \beta_i\gamma_i}
\]
we get
\begin{equation} \label{eq:Snah}
w(n,\xi,X) \le \sum_{i=1}^{\ell} \alpha_i \beta_i\gamma_i + o(1).
\end{equation}
On the other hand, considering the separable polynomial 
\[
Q:= \prod_{i=1}^{\ell} Q_i
\]
which satisfies
\[
\deg Q= \sum_{i=1}^{\ell} d_i
,\qquad |Q(\xi)| = \prod_{i=1}^{\ell} |Q_i(\xi)|
=\prod_{i=1}^{\ell} H_i^{-\gamma_i},
\]
we get with Lemma~\ref{feld} and using 
$H_Q\asymp_n \prod_{i=1}^{\ell} H_i$ by Gelfond's Lemma~\ref{gel} that it has
a root $\alpha$ with
\[
|\alpha-\xi| \ll_n H(\alpha)^{-1}\cdot \prod_{i=1}^{\ell} H_i^{-\gamma_i+d_i-1}= H(\alpha)^{-1}\cdot X^{-\sum_{i=1}^{\ell} \beta_i(\gamma_i-d_i+1)}
\]
thus
\begin{equation} \label{eq:Hans}
w^{\ast}(n,\xi,X) \ge \sum_{i=1}^{\ell} \beta_i\cdot (\gamma_i-d_i+1)-o(1).
\end{equation}
Comparing \eqref{eq:Snah}, \eqref{eq:Hans} we get 
\[
\kappa(n,\xi,X) \le  \sum_{i=1}^{\ell} \alpha_i \beta_i\gamma_i  -  \sum_{i=1}^{\ell} \beta_i\cdot (\gamma_i-d_i+1)+o(1). 
\]
We have thus proved the following 
local lemma.

\begin{lemma}  \label{lemar}
   For any large $X$ with induced parameters $\ell, d_i,\alpha_i,\beta_i,\gamma_i$ we have
   \[
   \kappa(n,\xi,X)\le
   \sum_{i=1}^{\ell} (\alpha_i-1)\beta_i\gamma_i + \beta_i(d_i-1)+o(1), \qquad \text{as}\; X\to\infty.
   \]
   Thus if
   we have
   \begin{equation} \label{eq:crit}
\sum_{i=1}^{\ell} (\alpha_i-1)\beta_i\gamma_i + \beta_i(d_i-1)\le  -1 + \sum_{i=1}^{\ell} \alpha_i d_i \le  n-1
\end{equation}
then 
\begin{equation} \label{eq:sti}
\kappa(n,\xi,X)\le n-1+o(1), \qquad \text{as}\; X\to\infty.
\end{equation}
\end{lemma}

\begin{remark}
    A similar argument can be applied
    to $Q:= \prod_{j\in J} Q_j$ for any subset $J\subseteq \{1,2,\ldots,\ell\}$. However it turns out one
    may restrict to the full set $J$.
\end{remark}

We prove the cases $n=4, n=5$ with Lemma~\ref{lemar} and Theorems~\ref{BuS}, \ref{Sch}.
 By Lemma~\ref{sepp} we only have to consider parameters $X$ for which $P_X$ is not separable. Moreover, we can assume that for all large $X$ it does not compose into linear factors, as otherwise by 
    Wirsing's Lemma~\ref{wirsing}
    we have $w_1(\xi)\ge \wo \ge n$ and
    the claim follows from Theorem~\ref{B}, (v).

    We may assume $P_X$ has exact degree
    $n$. Otherwise we could multiply it with the according power of the variable $T$ to get $\tilde{P}(T)=T^{t}P_X$ for $t=n-\
    \deg P_X$, which keeps the height unaffected $H_{\tilde{P}}=H_P$ and decreases its evaluation at $\xi$.
    i.e. $|\tilde{P}(\xi)|< |P(\xi)|$
    (as we may assume $\xi\in (0,1)$), contradicting the definition of $P_X$.

    We will use the notation of the 
    present~\S~\ref{nota} in the sequel.

\subsection{Proof for $n=4$}
    First let $n=4$.
    In view of the above restrictions, for $X$ any large parameter,
    there are two remaining cases. 

Case A: We have 
\[
\ell=2,\qquad P_X= Q_1^2 Q_2, \qquad d_1=1,\; d_2=2,
\]
that is $Q_1$ is linear and $Q_2$ irreducible quadratic.
Then \eqref{eq:side2} gives
\begin{equation}  \label{eq:bb}
    2\beta_1  + \beta_2\le 1+o(1),
\end{equation}
in particular 
\[
\beta_1\le \frac{1}{2}+o(1).
\]
The criterion \eqref{eq:crit} 
from Lemma~\ref{lemar} yields the sufficient condition
\[
\beta_1\gamma_1 + \beta_2\le n-1= 3.
\]
Using \eqref{eq:bb} this can be relaxed to $(\gamma_1-2)\beta_1+1+o(1)\le 3$ or equivalently
\[
\gamma_1 \le 2+\frac{2-o(1)}{\beta_1}-o(1).
\]
In particular $\gamma_1\le 6-\varepsilon$ for small enough $\varepsilon>0$ suffices
to have \eqref{eq:sti} for the corresponding (large) $X$ of Case A. Note
that by definition of $\wo$ this 
further implies
\begin{equation} \label{eq:agho}
\wo - w^{\ast}(n,\xi,X) \le n-1+o(1), \qquad X\to\infty,
\end{equation}
for these $X$.

However, if we let $\varepsilon=1/2$ and assume contrary $\gamma_1>6-\varepsilon$ for arbitrarily large $X$ inducing Case A then $w_1(\xi)\ge 6-1/2=5.5>4=n$, implying \eqref{eq:aaa} by Theorem~\ref{B}, (v). 
In particular we again have \eqref{eq:agho}. So \eqref{eq:agho} holds for all $X$ inducing Case A.

  Case B: We have 
    \[
    P_X= Q_1^2, \qquad d_1=2.
    \]
    Then $Q_1$ is quadratic irreducible and
    \[
    \ell=1, \qquad \alpha_1=\alpha=2,\qquad 0<\beta_1=\beta\le \frac{1}{2}+o(1), \qquad d_1=d=2, \qquad \gamma_1=\gamma.
    \]
    Now the sufficient criterion \eqref{eq:crit} of Lemma~\ref{lemar} becomes
\[
\beta(\gamma+1)\le 3
\]
or 
\begin{equation} \label{eq:ra}
   \gamma\le \frac{3}{\beta} - 1. 
\end{equation}
In particular
\begin{equation} \label{eq:inver}
\gamma\le 5-o(1)
\end{equation}
suffices to prove
\[
\kappa(4,\xi,X)=w(4,\xi,X)- w^{\ast}(4,\xi,X)\le 3+o(1)=n-1+o(1)
\]
for $X$ in question. We may 
thus assume otherwise
\[
\gamma > 5+\varepsilon = n+d-1+\varepsilon > n+d-1, \qquad \varepsilon>0.
\]
(While there is a small gap to the converse of \eqref{eq:inver},
indeed if against our claim \eqref{eq:aaa} we have $w(4,\xi,X)- w^{\ast}(4,\xi,X)=n-1+\delta$ for some $\xi$ and a
strictly positive $\delta=\delta(\xi)>0$ for certain large $X$ uniformly, then the above argument easily yields we may take $\varepsilon=\varepsilon(\delta)=\delta/2>0$, we omit details.) 
If this happens for arbitrarily large $X$, then 
by Theorem~\ref{BuS} we have
\[
\wo = \min\{ \wo, w_d(\xi) \} \le n+d-1=n+1.
\]
So it suffices to show
\begin{equation}  \label{eq:lllast}
w^{\ast}(n,\xi,X) \ge 2-o(1), \qquad X\to\infty,
\end{equation}
for $X$ inducing Case B, to deduce that for such $X$ again \eqref{eq:agho}
holds. Assume this is true. As we have observed \eqref{eq:agho} holds for any large $X$ inducing Case A as well, we can pass
to the lower limit to
indeed deduce $\wo-\ws\le (n+1)-2= n-1$, i.e. \eqref{eq:aaa}.

However, for $X$ inducing Case B, by Feldman's Lemma~\ref{feld} the quadratic $Q_1=Q$ has a root $\alpha$ with
\[
|\alpha-\xi| \ll H_Q^{-\gamma} =  H(\alpha)^{-1}
H_Q^{-(\gamma-1)} =  H(\alpha)^{-1}
X^{-\beta (\gamma-1)}.
\]
Hence if
\[
\beta(\gamma-1)\ge 2-o(1), \qquad X\to\infty,
\]
or equivalently
\begin{equation} \label{eq:re}
    \gamma\ge 1+\frac{2-o(1)}{\beta}
\end{equation}
we have \eqref{eq:lllast} and are done. So combining the criteria \eqref{eq:ra}, \eqref{eq:re}, the only problematic case
is
\[
 \frac{3}{\beta}-1 < \gamma < 1+\frac{2-\delta}{\beta},
\]
for some fixed $\delta>0$ and 
certain arbitrarily large $X$ inducing Case B.
But this implies $\beta>1/2+\delta/2$, contradiction for large $X$.

\begin{remark}
    We stress again that Theorem~\ref{BuS} and Theorem~\ref{Sch}
    were vital ingredients of the proof
    of both Cases A,B. Indeed, the stronger claim $\kappa(4,\xi,X)\le n-1+o(1)$ is false in either case in general, by a generalization of Theorem~\ref{co}. The same will be true likewise for $n=5$ below.
\end{remark}

\subsection{Proof for $n=5$}
We now prove the claim for $n=5$.
Here by the initial argument 
as for $n=4$
we have now four remaining cases:

Case A: 
\[
P_X= Q_1^2 Q_2, \qquad d_1=1,\quad 
d_2=3.
\]
Then 
\[
2\beta_1 + \beta_2 \le 1+o(1)
\]
in particular
\[
\beta_1\le \frac{1}{2}+o(1)
\]
and \eqref{eq:crit} becomes
\[
\beta_1 \gamma_1 + 2 \beta_2 \le n-1 = 4.
\]
So a sufficient condition is
\[
\gamma_1 \le \frac{2-o(1)}{\beta_1} + 4.
\]
Hence
\[
\gamma_1\le 7.5
\]
suffices for large $X$. But if for arbitrarily large such $X$ 
\[
\gamma_1 >7.5 > 5 =n
\]
then $w_1(\xi)\ge n$ and we conclude  \eqref{eq:agho} for $X$ in question
by the same arguments as for $n=4$.

Case B: 
\[
P_X= Q_1^3 Q_2, \qquad d_1=1,\; d_2=2 
\]
Then
\[
3\beta_1+\beta_2 \le 1+o(1)
\]
thus
\[
\beta_1\le \frac{1}{3}+o(1)
\]
and
\eqref{eq:crit} becomes
\[
2\beta_1 \gamma_1 + \beta_2 \le 4= n-1.
\]
Equivalently 
\[
\beta_1 (2\gamma_1-3)\le 3-o(1),\qquad \gamma_1\le \frac{\frac{3-o(1)}{1/3+o(1)}+3}{2}=6-o(1),
\]
so
\[
\gamma_1 \le 6-\varepsilon, \qquad \varepsilon>0,
\]
suffices for large enough $X$.
But if for arbitrarily large $X$
and $\varepsilon=1/2$ we have
\[
\gamma_1>6-\varepsilon=5.5> 5=n
\]
then again $w_1(\xi)\ge n$ and we again conclude that \eqref{eq:agho} holds.

Case C:
\[
P_X= Q_1^2 Q_2 Q_3, \qquad d_1=d_2=1,\; d_3=2.
\]
Then 
\[
2\beta_1 + \beta_2+ \beta_3\le 1+o(1)
\]
in particular
\[
\beta_1\le \frac{1}{2}+o(1)
\]
and \eqref{eq:crit} yields the criterion
\[
\beta_1 \gamma_1 + \beta_3 \le 4 = n-1,
\]
hence as $\beta_2\ge 0$ we get
\[
\beta_1(\gamma_1-2) +1 \le 4, \qquad
\gamma_1\le \frac{3}{\beta_1} + 2
\]
as sufficient criterion. Now if conversely for arbitrarily large $X$ we have
\[
\gamma_1 > \frac{3}{\beta_1} + 2 - o(1) \ge 8-o(1),
\]
then again $w_1(\xi)\ge 8-\varepsilon=7.5>n$ and we conclude again that \eqref{eq:agho} holds.

Case D: 
\[
P_X= Q_1^2 Q_2, \qquad d_1=2,\; d_2=1.
\]
Then again
\[
2\beta_1 + \beta_2 \le 1+o(1)
\]
in particular
\[
\beta_1\le \frac{1}{2}+o(1)
\]
and \eqref{eq:crit} becomes
\[
\beta_1 \gamma_1 + \beta_1= \beta_1(\gamma_1+1) \le n-1 = 4
\]
so sufficient is
\[
\gamma_1 \le \frac{4}{\beta_1} - 1.
\]
Hence
\[
\gamma_1\le 7-\varepsilon, 
\]
for uniform $\varepsilon>0$ and large $X$ as in Case D suffices to prove
\[
w(5,\xi,X)- w^{\ast}(5,\xi,X)\le n-1+o(1)
\]
for all such $X$. But if for $\varepsilon=1/2$ and arbitrarily large $X$ 
\[
\gamma_1 >7-\varepsilon=6.5> 6 = n+d_1-1
\]
holds, then by Theorem~\ref{BuS} we 
have
\[
\widehat{w}_n(\xi)= \min\{ \widehat{w}_n(\xi), w_{d_1}(\xi) \}
\le n+d_1-1=n+1.
\]
So we
only have to prove that
\[
w^{\ast}(5,\xi,X)\ge 2-o(1), \qquad X\to\infty,
\]
for $X$ inducing Case D, to conclude 
first \eqref{eq:agho} and finally
\eqref{eq:aaa} by the same arguments as for $n=4$.
By Feldman's Lemma~\ref{feld}, for this again with the same argument as in Case B of $n=4$ it suffices to have
\[
\gamma_1 \ge 1 + \frac{2}{\beta_1}.
\]
So the only bad case is
\[
\frac{4}{\beta_1} - 1 < \gamma_1 < 1 + \frac{2}{\beta_1}.
\]
But this implies $\beta_1>1$, contradiction.

\section{Proof of Theorem~\ref{co}}

    We only need to prove \eqref{eq:apply}, then \eqref{eq:apply2} is just the case $k=n$ and
    \eqref{eq:apply3} follows in turn for $w_1(\xi)=\infty$.
    Let $Q(T)=aT-b$ with $a>0$ and $H_Q>1$ be very small at $\xi$, say
    \[
    |Q(\xi)|= H_Q^{-\lambda}
    \]
    for some $\lambda>n+k-1$. Clearly we can assume $|b|\asymp_{\xi} a$ so this implies
    \begin{equation} \label{eq:END}
         |\xi-b/a|= \frac{|Q(\xi)|}{a} \asymp_{\xi} H_Q^{-\lambda-1}.
    \end{equation}
    Consider the parameter
    \[
    X= H_{Q^k}\asymp_n H_Q^{k},
    \]
    where we used Gelfond's Lemma~\ref{gel}.
    Consider $R=Q^k$, which has degree $k\le n$ and $H_R= X$. Moreover
    \[
    |R(\xi)| = H_Q^{-k\lambda}\ll_{n,\lambda} X^{-\lambda}.
    \]
    Hence
     \[
    w(n,\xi,X)\ge \lambda-o(\lambda), 
    \]
    as $H_Q\to\infty$.

    On the other hand, any other algebraic number $\eta$ of degree at most $n$ and height at most $X$ has distance from $b/a$ at least
    \[
    |\eta - b/a| \gg_n X^{-1} H_Q^{-n}\gg_n X^{-1-n/k}
    \]
    by Liouville inequality Theorem~\ref{liouville}. So by $\lambda>n+k-1$, this implies by \eqref{eq:END} and triangle inequality for large enough $X$
    \[
    |\xi- \eta|\ge |b/a-\eta|-|\xi-b/a|\ge \frac{1}{2} |b/a-\eta|\gg_n X^{-1-n/k}.
    \]
    So the smallest value comes from $b/a$ the root of $Q$,
    for which we use \eqref{eq:END} and 
    estimating its right hand side as $\gg H(b/a)^{-1} X^{-\lambda/k-o(\lambda)}$ we
    get
    \[
    w^{\ast}(n,\xi,X)\le  \frac{\lambda}{k}+o(\lambda).
    \]
   Combining the two bounds we get
    \[
    \kappa(n,\xi,X)\ge \lambda- \frac{\lambda}{k}-o(\lambda)= \left(1-\frac{1}{k}-o(1)\right) \lambda.
    \]
    As $\lambda$ can be chosen arbitrarily close to $w_1(\xi)$
    and by assumption $w_1(\xi)>n+k-1$, we can find infinitely many such $Q$ and thus arbitrarily large $X$, and the claim follows.

    \section{Proof of Theorem~\ref{liou}}

    Let us first recall that $\widehat{w}_2(\xi)=2$ is implied by the hypothesis $w_1(\xi)>3$ as mentioned
    in the paragraph following Theorem~\ref{liou}.

    For the most left inequality of
    \eqref{eq:zwe},
    we first prove an auxiliary lemma 
    for approximation to a single number
    using a standard determinant argument.

    \begin{lemma} \label{lemur}
        Let $P(T)=cT+d$ of height $H_P>1$ be a linear integer
        polynomial best approximation
        for $n=1$ (essentially this means $-d/c$ is a convergent to $\xi$). Let $\lambda\ge 1$ be determined by
        \[
        |P(\xi)|= H_P^{-\lambda}.
        \]
        Let $R(T)=aT+b$ be
        the linear integer polynomial
        that minimizes
        $|Q(\xi)|$ among all linear integer polynomials $Q(T)$ of height $H_Q\le C\cdot H_P^{\lambda}$
        for fixed $C\in (0,\frac{1}{2})$ that are not a scalar
        multiple of $P$. Then
        \[
        |R(\xi)| \asymp_C H_P^{-1}.
        \]
        The implied constants are absolute.
    \end{lemma}

    \begin{proof}
     By well known estimates for continued fractions, for 
     \[
        |R(\xi)| \ll H_P^{-1}
        \]
        it suffices to take $R(T)=Q(T)=aT+b$ 
        the best approximation polynomial
        for the parameter $Y=CH_P$. This
        is clearly not a multiple of $P$
        as $C<1$ and $P$ has a primitive coefficient vector since it is a best approximation polynomial.
         Moreover as $\lambda\ge 1$ clearly $H_Q\le C\cdot H_P\le C\cdot H_P^{\lambda}$,
        and by Dirichlet's Theorem and definition of $R$ indeed
        \[
        H_R\le C\cdot H_P, \qquad |R(\xi)|\le |Q(\xi)| \le (C\cdot H_P)^{-1}\ll_C H_P^{-1}.
        \]
        (Usually $R$ corresponds to
        the convergent $-b/a$ of $\xi$ preceding $-d/c$ in the continued fraction algorithm). 
        
        For the reverse inequality, 
        considering the 
        matrix formed
        by the integer coefficient vectors of $P,R$ gives by linear independence and multi-linearity of the determinant
        \[
        1\le \begin{vmatrix}
     a & b \\ 
     c & d 
\end{vmatrix} 
= \begin{vmatrix}
     a & R(\xi) \\ 
     c & P(\xi) 
\end{vmatrix}
\le |P(\xi)| H_R + |R(\xi)|H_P\le C+ |R(\xi)| H_P
        \]
        and the claim follows after a short rearrangement using $0<C<1/2$.
    \end{proof}

Now we prove the theorem.
Let $P, R$ be as in Lemma~\ref{lemur}, for arbitrary but fixed $C\in (0,1/2)$,
and assume $\lambda>3$. By Lemma~\ref{lemur}
they satisfy
\[
|P(\xi)|= H_P^{-\lambda}, \quad |R(\xi)|\asymp H_P^{-1}.
\]
Hence the polynomials
\[
V_1:= P^2, \qquad V_2:= PR
\]
are quadratic and satisfy
\begin{equation}  \label{eq:viaeq}
\max_{i=1,2} H_{V_i}\ll H_P^2, \qquad
|V_1(\xi)|= H_P^{-2\lambda}, \quad |V_2(\xi)|\asymp H_P^{-\lambda-1}.
\end{equation}
Take the parameter
\[
X= H_P^{\lambda-1-\varepsilon}
\]
for small $\varepsilon\in (0,(\lambda-3)/2)$ and consider all
linear or quadratic integer polynomials $V_3$ of height at most $H_{V_3}\le X$.
Clearly any polynomial in $\rm{span} \{V_1\}$ is not separable, hence 
for our exponent we can
restrict to the complementary set of polynomials. We show that any such $V_3\notin \rm{span} \{ V_1\}$ with $H_{V_3}\le X$ satisfies
\begin{equation} \label{eq:tru}
|V_3(\xi)| \gg H_P^{-\lambda-1}.
\end{equation}
If this is true then
\[
w_{sep}(2,\xi,X)\le \frac{\lambda+1+o(1)}{\lambda-1-\varepsilon}= 1+\frac{2+\varepsilon+o(1)}{\lambda-1-\varepsilon}.
\]
The claim \eqref{eq:zwe} of the theorem follows with $H_P\to\infty$ as we may take $\varepsilon$ arbitrarily small and $\lambda$
arbitrarily close to $w_1(\xi)$.
The latter specialization for Liouville
numbers follows as then the upper bound becomes $1$, while by considering linear polynomials only we get the reverse inequality, to conclude 
\[
1\ge \widehat{w}_{2,sep}(\xi)\ge\widehat{w}_{2,irr}(\xi)\ge  \widehat{w}_{1,irr}(\xi)=
\widehat{w}_1(\xi)=1.
\]
So the inequalities must be equalities.
To prove \eqref{eq:tru}, we distinguish two cases.

Case 1: $V_3\in \rm{span}\{ V_1, V_2\}\setminus \rm{span} \{V_1\}$.
This means $V_3=PS$ with linear $S\notin \rm{span} \{P\}$. 
Moreover by Gelfond's Lemma~\ref{gel}
we have $H_S\ll H_{V_3}/H_P \le X/H_P= H_P^{\lambda-2-\varepsilon}$, hence
we may assume $H_S < C\cdot H_P^{\lambda}$
if we assume $H_P$ is large enough. By Lemma~\ref{lemur} we conclude
\[
|S(\xi)| \gg H_P^{-1}
\]
and thus
\[
|V_3(\xi)|= |P(\xi)|\cdot |S(\xi)| \gg H_P^{-\lambda-1},
\]
so \eqref{eq:tru} holds.

Case 2: $V_3\notin \rm{span}\{ V_1, V_2\}$. So assume contrary to \eqref{eq:tru}
that some linear or quadratic integer polynomial $V_3$ (in fact it is not hard to show that $V_3$ must be irreducible quadratic) outside the span of $\{ V_1, V_2\}$ 
satisfies the estimates
\begin{equation}  \label{eq:evia}
H_{V_3}\le H_P^{\lambda-1-\varepsilon}, \qquad    |V_3(\xi)| \le H_P^{-\lambda-1}.  
\end{equation}
We proceed similarly as in the proof
of Lemma~\ref{lemur}. Writing 
the coefficient vectors of $V_i(\xi)$
as $V_{i}(\xi)=v_{i,0}+v_{i,1}\xi+v_{i,2}\xi^2$, $1\le i\le 3$, from
linear independence and multi-linearity of the determinant we get
\[
1\le \begin{vmatrix}
     v_{1,2} & v_{1,1} & v_{1,0}\\ 
     v_{2,2} & v_{2,1} & v_{2,0}\\
     v_{3,2} & v_{3,1} & v_{3,0} 
\end{vmatrix} =
\begin{vmatrix}
     v_{1,2} & v_{1,1} & V_1(\xi)\\ 
     v_{2,2} & v_{2,1} & V_2(\xi)\\
     v_{3,2} & v_{3,1} & V_{3}(\xi) 
\end{vmatrix}.
\]
Now expanding the right hand side determinant and using \eqref{eq:viaeq} and \eqref{eq:evia}
to estimate $|v_{i,j}|$ and
$|V_i(\xi)|$ from above, we can estimate
its absolute as $\ll H_P^{-\delta}$ for $\delta=\min\{ \varepsilon, \lambda-3\}>0$ using $\lambda>3$. Together with the lower bound $1$ we
derive a contradiction for large enough $H_P$.

\section{Proof of Theorem~\ref{n=2} }

    For \eqref{eq:EST1},
      note that
      the weaker estimate
       \[
    \underline{\kappa}_2(\xi) \le
    w_2(\xi)-w_2^{\ast}(\xi)\le w_2(\xi)-\widehat{w}_2(\xi)^2 + \widehat{w}_2(\xi) \le w_2(\xi)-2
    \]
      follows directly from combining~\cite[Theorem~2]{mos}
      (see also~\cite{die2}) and \eqref{eq:Con}.
       For the stronger claim, it follows more precisely from the proof of~\cite{mos} that there are infinitely many best approximations satisfying
      \[
      |P^{\prime}(\xi)| \ge |P(\xi)|\cdot H_P^{ \widehat{w}_2(\xi)^2 - \widehat{w}_2(\xi)+1  }
      \]
      which implies $P$ has a root
      $\alpha$ with
      \[
      |\xi-\alpha| \ll H(\alpha)^{-1}\cdot H_P^{ -(\widehat{w}_2(\xi)^2 - \widehat{w}_2(\xi))+o(1)  }.
      \]
Moreover, a result of Jarn\'ik~\cite{jarnik, jarnik2} implies that $P$ remains a best approximations up to a parameter of size $X\ge H_P^{\sigma-o(1)}$ for $\sigma= \widehat{w}_2(\xi)-1$.
       Indeed this happens whenever
       $P$ is linearly independent with
       its preceding and successive
       minimal polynomial, which occurs for our $P$ above. 
       Hence
        \[
    |\xi-\alpha| \ll H(\alpha)^{-1} H_P^{ -(\widehat{w}_2(\xi)^2 - \widehat{w}_2(\xi))+o(1)  }= H(\alpha)^{-1} X^{-\frac{\widehat{w}_2(\xi)^2 - \widehat{w}_2(\xi)}{\sigma} +o(1)}.
    \]
        In other words
    \[
    w^{\ast}(2,\xi,X)\ge \frac{\widehat{w}_2(\xi)^2 - \widehat{w}_2(\xi)}{\sigma}-o(1),
    \]
    hence since clearly $w(2,\xi,X)\le w_2(\xi)/\sigma+o(1)$ we get
    \[
    \kappa(2,\xi,X)\le w(2,\xi,X) - \frac{\widehat{w}_2(\xi)^2 - \widehat{w}_2(\xi)}{\sigma}+o(1)\le \frac{w_2(\xi)}{\widehat{w}_2(\xi)-1} - \widehat{w}_2(\xi)+o(1).
    \]
     The claim follows as there are arbitrarily large such $X$.
    
       For \eqref{eq:EST2},
       we again use $P,X$ as above and for simplicity write $w=w(2,\xi,X)$.
By Lemma~\ref{feld} below, $P$ has a root $\alpha$ satisfying
    \[
    |\xi-\alpha| \ll |P(\xi)|=X^{-w}= H(\alpha)^{-1} X^{-w} H_P= H(\alpha)^{-1} X^{-\frac{\sigma w-1}{\sigma} }.
    \]
    In other words
    \[
    w^{\ast}(2,\xi,X)\ge w-\frac{1}{\sigma}-o(1).
    \]
        Rearranging, for such $X$ we have
        \[
        \kappa(2,\xi,X)=w(2,\xi,X)-w^{\ast}(2,\xi,X)\le \frac{1}{\sigma}+o(1)\le \frac{1}{\widehat{w}_2(\xi)-1}+o(1),
        \]
        so the claim follows.

        \vspace{0.5cm}

        {\em The author thanks the referee for many useful suggestions, especially for
        finding simplifications in the proofs of Lemma 9.1 and 
        Theorem 3.5.
        The author further thanks Yann Bugeaud for helpful remarks, especially for bringing to the author's attention a question that
        led to Theorem~3.5.  }


\begin{thebibliography}{99}
	
	
	
	
	 \bibitem{badsch} D. Badziahin, J. Schleischitz. An improved bound in Wirsing's problem. {\em Trans. Amer. Math. Soc.} 374 (2021), no. 3, 1847--1861.

 \bibitem{bere} V. Beresnevich. Rational points near manifolds and metric Diophantine approximation. {\em Ann. of Math.} (2) 175 (2012), no. 1, 187--235. 


  \bibitem{bernik}  V. I. Bernik. Application of the Hausdorff dimension in the theory of Diophantine approximations. (Russian) {\em Acta Arith.} 42 (1983), no. 3, 219--253. 
	 
	\bibitem{b2003} Y. Bugeaud. Mahler's classification of numbers compared with Koksma's. {\em Acta Arith.} 110 (2003), no. 1, 89--105.

  \bibitem{bugbuch} Y. Bugeaud.
	 Approximation by algebraic numbers. {\em Cambridge Tracts in Mathematics}
	 160, Cambridge University Press, Cambridge, 2004.
	 
	 \bibitem{bugdraft} Y. Bugeaud. Exponents of Diophantine approximation. {\em Dynamics and analytic number theory}, 96--135, London Math. Soc. Lecture Note Ser., 437, Cambridge Univ. Press, Cambridge, 2016.
	 
	

  \bibitem{buglau} Y. Bugeaud, M. Laurent. Exponents of Diophantine approximation and Sturmian continued fractions. {\em Ann. Inst. Fourier (Grenoble)} 55 (2005), no. 3, 773--804.
	 
	 \bibitem{buschlei} Y. Bugeaud, J. Schleischitz. On uniform approximation to real numbers.
	 {\em Acta Arith.} 175 (2016), 255--268.

  \bibitem{die2} H. Davenport, W. M. Schmidt. Approximation to real numbers by quadratic
irrationalities. {\em Acta Arith.} 13 (1967), 169--176. 

  \bibitem{davsh} H. Davenport, W. Schmidt.
  Approximation to real numbers by algebraic integers. {\em Acta Arith.} 15 (1969), 393--416.

  \bibitem{jarnik} V. Jarn\'ik. Une remarque sur les approximations diophantiennes lin\'eaires.
  {\em Acta Sci. Math. Szeged} 12 (1950), pars B, 82--86

  \bibitem{jarnik2} V. Jarn\'ik. Contribution to the theory of of homogeneous linear Diophantine approximations. {\em Czechoslovak Math. J.} 79 (1954), no. 4, 330--353. 

   \bibitem{khint} A.Y. Khintchine. \"Uber eine Klasse linearer diophantischer Approximationen. {\em Rend. Circ. Mat. Palermo} 50 (1926), 706--714. 

  \bibitem{mamo} A. Marnat, N.G. Moshchevitin.  An optimal bound for the ratio between ordinary and uniform exponents of Diophantine approximation.
{\em Mathematika} 66 (2020), no. 3, 818--854.

  \bibitem{mos} N.G. Moshchevitin. Diophantine exponents for systems of linear forms in two variables. {\em Acta Sci. Math. (Szeged)} 79 (2013), no. 1--2, 347--367.


  \bibitem{poe} A. Poels. On approximation to a real number by algebraic numbers of
bounded degree. {\em arXiv: 2405.08341}.

  \bibitem{poels} A. Poels. On uniform polynomial approximation. {\em arXiv: 2405.07219}.

  
	 
	\bibitem{ichdeb} J. Schleischitz. Two estimates concerning classical diophantine approximation constants. {\em Publ. Math. Debrecen} 84 (2014), no. 3--4, 415--437.

 
     \bibitem{mathematika} J. Schleischitz.
     On the spectrum of Diophantine approximation constants. {\em Mathematika} 62 (2016), no. 1, 79--100.
     
     \bibitem{ichacta} J. Schleischitz. Uniform Diophantine approximation and best approximation polynomials. {\em Acta Arith.} 185 (2018), no. 3, 249–-274.
     
     \bibitem{moscj} J. Schleischitz. Diophantine approximation in prescribed degree. {\em Mosc. Math. J.}  18 (2018), no. 3, 491--516.
     
     \bibitem{equprin} J. Schleischitz. An equivalence principle between polynomial and simultaneous Diophantine approximation. {\em  Ann. Sc. Norm. Super. Pisa Cl. Sci.} (5) 21 (2020), 1063--1085. 

     \bibitem{newp} J. Schleischitz.
     On Wirsing's problem in small exact degree. {\em to appear in Mosc. Math. J,. arXiv: 2108.01484}.

     \bibitem{icharc} J. Schleischitz.
     Uniform dual approximation to Veronese curves in small dimension. {\em arXiv: 2405.10086}.

     \bibitem{sprind} V.G. 
     Sprind\v{z}uk. Proof of the hypothesis of Mahler on the measure of the set of
$S$-numbers. {\em Izv. Akad. Nauk SSSR (ser. mat.)} 29, 379--436 (in Russian).
     
     \bibitem{wirsing} E. Wirsing. Approximation mit algebraischen Zahlen beschr\"ankten Grades.
     {\em J. Reine Angew. Math.} 206 (1961), 67--77.











\end{thebibliography}
\end{document}